\documentclass{article}%
\usepackage{utopia}
\usepackage{amssymb,latexsym}
\usepackage{amsfonts}
\usepackage{amsmath}
\usepackage{longtable}
\usepackage[colorlinks=true, pdfstartview=FitV, linkcolor=blue, citecolor=red,
urlcolor=blue]{hyperref}
\usepackage{color}
\usepackage[compress]{cite}
\usepackage{amssymb}
\usepackage{graphicx}%
\providecommand{\U}[1]{\protect\rule{.1in}{.1in}}
\newtheorem{theorem}{Theorem}

\newtheorem{corollary}[theorem]{Corollary}

\newtheorem{lemma}[theorem]{Lemma}

\newtheorem{remark}[theorem]{Remark}

\newenvironment{proof}[1][Proof]{\noindent\textbf{#1.} }{\ \rule{0.5em}{0.5em}}
\allowdisplaybreaks
\begin{document}

\title{Euler sums of generalized harmonic numbers and connected extensions}
\author{M\"{u}m\"{u}n CAN, Levent KARGIN, Ayhan D\.IL, G\"{u}ltekin SOYLU\\
Department of Mathematics, Akdeniz University, Antalya, Turkey.}
\date{}
\maketitle

\begin{abstract}
This paper presents the evaluation of the Euler sums of generalized
hyperharmonic numbers $H_{n}^{\left(  p,q\right)  }$
\[
\zeta_{H^{\left(  p,q\right)  }}\left(  r\right)  =\sum\limits_{n=1}^{\infty
}\dfrac{H_{n}^{\left(  p,q\right)  }}{n^{r}}%
\]
in terms of the famous Euler sums of generalized harmonic numbers. Moreover,
several infinite series, whose terms consist of certain harmonic numbers and
reciprocal binomial coefficients, are evaluated in terms of Riemann zeta
values.
%\keywords{Harmonic numbers, hyperharmonic numbers, generalized
%harmonic numbers, Euler sums, Riemann zeta function, Stirling numbers.}

\end{abstract}

\section{Introduction}

The classical Euler sum $\zeta_{H}\left(  r\right)  $ is the following
Dirichlet series
\[
\zeta_{H}\left(  r\right)  =\sum\limits_{n=1}^{\infty}\dfrac{H_{n}}{n^{r}},
\]
where $H_{n}$ is the $n$th harmonic number. This series is also known as the
harmonic zeta function. The famous Euler's identity for this sum is
\cite{CS,LE,Ni}%
\begin{equation}
2\zeta_{H}\left(  r\right)  =\left(  r+2\right)  \zeta\left(  r+1\right)
-\sum_{j=1}^{r-2}\zeta\left(  r-j\right)  \zeta\left(  j+1\right)  ,\text{
}r\in\mathbb{N}\backslash\left\{  1\right\}  , \label{ES}%
\end{equation}
where $\zeta\left(  r\right)  $ is the classical Riemann zeta function (for
more details, see for instance \cite{SrC}). Many generalizations of Euler sums
(the so called Euler-type sums) are given using generalizations of harmonic
numbers (see \cite{AC,BA,BJG,Bo,BK,Br,DiMeCe,KK,S1,S,Xu1,XYS,YW,Y,Y1,WL,ZCB}).
Evaluation of Euler-type sums and construction of closed forms are active
fields of study in analytical number theory. Furthermore \cite{BJG,Bo, Br, C,
C2} are some of the studies that make this area interesting in the sense that
Euler sums have potential applications in quantum field theory and knot
theory, especially in evaluation of Feynman diagrams.

Euler actually considered also the more general form \cite{BB, CB, LE,Ni}%
\begin{equation}
\zeta_{H^{\left(  p\right)  }}\left(  m\right)  =\sum_{n=1}^{\infty}%
\frac{H_{n}^{\left(  p\right)  }}{n^{m}}, \label{g1}%
\end{equation}
where $H_{n}^{\left(  p\right)  }$ defined by
\[
H_{n}^{\left(  p\right)  }=1+\frac{1}{2^{p}}+\frac{1}{3^{p}}+\cdots+\frac
{1}{n^{p}},\text{ }\left(  p\in\mathbb{Z},n\in\mathbb{N}\right)  ,
\]
is the $n$th partial sum of $\zeta\left(  p\right)  $ and is called the $n$th
generalized harmonic number for $p>1$. In particular, $H_{0}^{\left(
p\right)  }=0$ and $H_{n}^{\left(  1\right)  }=H_{n}$, the $n$th harmonic
number. When $p\leq0$ it is called sum of powers of integers.

One of the most important issues here is to write Euler-type sums as
combinations of the Riemann zeta function as in (\ref{ES}). This problem has
remained important for various Euler-type sums from the era of Euler to the
present day. It's shown by Euler himself that, the cases of $p=1$, $p=q$,
$p+q$ odd, and for special pairs $(p,q)\in\{(2,4),(4,2)\}$, the sums of the
form (\ref{g1}) have evaluations in terms of the Riemann zeta function (see
\cite{BB, CB, LE,Ni}). There is a very comprehensive literature on this
subject, both theoretical and numerical (\cite{A, Borwein, Bo,Br,CB,CS,DB,
DiMeCe, MD, S,Xu1,Xu, XYS,XZZ, YW, WL}). One of these results; the Euler
identity (\ref{ES}) was further extended in the works of Borwein et al.
\cite{Borwein} and Huard et al. \cite{HuWY}. For odd weight $N\geq3$ and
$p=1,2,\ldots,N-2,$ we have \cite[Theorem 1]{HuWY} (or \cite[p. 278]%
{Borwein})
\begin{align}
\zeta_{H^{\left(  p\right)  }}\left(  N-p\right)   &  =\left(  -1\right)
^{p}\sum_{j=0}^{\left[  \left(  N-p-1\right)  /2\right]  }\binom{N-2j-1}%
{p-1}\zeta\left(  N-2j\right)  \zeta\left(  2j\right) \label{GES}\\
&  +\left(  -1\right)  ^{p}\sum_{j=0}^{\left[  p/2\right]  }\binom
{N-2j-1}{N-p-1}\zeta\left(  N-2j\right)  \zeta\left(  2j\right)  -\zeta\left(
0\right)  \zeta\left(  N\right)  .\nonumber
\end{align}
Moreover, these so called "linear Euler sums" satisfy a simple reflection
formula%
\begin{equation}
\zeta_{H^{\left(  p\right)  }}\left(  r\right)  +\zeta_{H^{\left(  r\right)
}}\left(  p\right)  =\zeta\left(  p+r\right)  +\zeta\left(  p\right)
\zeta\left(  r\right)  . \label{rp}%
\end{equation}

Considering nested partial sums of the harmonic numbers, Conway and Guy
\cite{CG} introduced hyperharmonic numbers for an integer $r>1$ as%
\[
h_{n}^{(r)}=\sum_{k=1}^{n}h_{k}^{(r-1)},\text{ }n\in\mathbb{N},
\]
with $h_{n}^{(0)}=1/n,$ $h_{n}^{(1)}=H_{n}$ and $h_{0}^{(r)}=0.$ Hyperharmonic
numbers are also important because they build a step in the transition to the
multiple zeta functions (see \cite{KK, Xu}). Dil and Boyadzhiev \cite{DB}
extended the Euler's identity (\ref{ES}) to the Euler sums of the
hyperharmonic numbers:%
\begin{equation}
\zeta_{h^{\left(  q\right)  }}\left(  r\right)  =\sum_{n=1}^{\infty}%
\frac{h_{n}^{\left(  q\right)  }}{n^{r}},\text{ }(r>q) ,\label{g2}%
\end{equation}
as
\begin{align}
\zeta_{h^{\left(  q\right)  }}\left(  r\right)   & =\frac{1}{\left(
q-1\right)  !}\sum_{k=1}^{q}%
%TCIMACRO{\QATOPD{[}{]}{q}{k}}%
%BeginExpansion
\genfrac{[}{]}{0pt}{}{q}{k}%
%EndExpansion
\label{HES}\\
& \times\left\{  \zeta_{H}\left(  r-k+1\right)  -H_{q-1}\zeta\left(
r-k+1\right)  +\sum_{j=1}^{q-1}\mu\left(  r-k+1,j\right)  \right\}  ,
\nonumber
\end{align}
where $%
%TCIMACRO{\QATOPD{[}{]}{q}{k}}%
%BeginExpansion
\genfrac{[}{]}{0pt}{}{q}{k}%
%EndExpansion
$ is the Stirling number of the first kind and%
\begin{equation}
\mu\left(  r,j\right)  =\sum_{n=1}^{\infty}\frac{1}{n^{r}\left(  n+j\right)
}=\sum_{k=1}^{r-1}\frac{\left(  -1\right)  ^{k-1}}{j^{k}}\zeta\left(
r+1-k\right)  +\left(  -1\right)  ^{r-1}\frac{H_{j}}{j^{r}}. \label{mu}%
\end{equation}
Formula (\ref{HES}) was the general form of the results obtained for some
special values of $q$ and $r$ in the study of \cite{MD}.

Studies on evaluating Euler sums (\ref{g1}) and (\ref{g2}) in terms of Riemann
zeta values $\zeta\left(  m\right)  $ have motivated researchers to find
representations harmonic number series of the forms%
\[
\sum_{n=1}^{\infty}\frac{H_{n}^{\left(  p\right)  }}{\left(  n+m\right)
^{r}\dbinom{n+m+l}{l}},\text{ }\sum_{n=1}^{\infty}\frac{h_{n}^{\left(
q\right)  }}{n\dbinom{n+q}{q}}.
\]
It has been shown that some families of these type of series can be evaluated
in terms of Euler sums and Riemann zeta values (see for example for $m=0,$
$p=1,$ $r\in\left\{  0,1\right\}  $ \cite{S2,S4,SC}, for $m=r=0$ \cite{S3},
for $m=0$ \cite{S2,XZZ}$,$ for $m>0,$ $r=1,$ $p\in\left\{  1,2\right\}  $
\cite{S5,SS} and for the series involving hyperharmonic numbers \cite{DB}). We
would like to emphasize that in some studies these type of series have been
expressed in terms of hypergeometric series \cite{Choi,CS,MD,S5,S,SS}.

In this work we mainly concentrate on generalized hyperharmonic numbers
defined as (see \cite{DiMeCe})%
\begin{equation}
H_{n}^{\left(  p,r\right)  }=\sum_{k=1}^{n}H_{k}^{\left(  p,r-1\right)
},\text{ }\left(  p\in\mathbb{Z},r\in\mathbb{N}\right)  ,\label{0}%
\end{equation}
with $H_{n}^{\left(  p,0\right)  }=1/n^{p}$. These are a unified extension of
generalized harmonic numbers and hyperharmonic numbers:%
\[
H_{n}^{\left(  p,1\right)  }=H_{n}^{\left(  p\right)  }\text{ and }%
H_{n}^{\left(  1,r\right)  }=h_{n}^{(r)}.
\]

The main objective of this study is the evaluation of Euler sums of
generalized hyperharmonic numbers%
\[
\zeta_{H^{\left(  p,q\right)  }}\left(  r\right)  =\sum\limits_{n=1}^{\infty
}\dfrac{H_{n}^{\left(  p,q\right)  }}{n^{r}}.
\]
A step towards the solution of this problem is taken in \cite{DiMeCe}.
However, the recurrence used by the authors did not return a closed formula.
Without an available closed formula, they listed only the following few
special cases%
\begin{align*}
\zeta_{H^{\left(  p,2\right)  }}\left(  r\right)   &  =\zeta_{H^{\left(
r-1\right)  }}\left(  p\right)  -\zeta_{H^{\left(  r\right)  }}\left(
p-1\right)  +\zeta_{H^{\left(  r\right)  }}\left(  p\right)  ,\\
2\zeta_{H^{\left(  p,3\right)  }}\left(  r\right)   &  =2\zeta_{H^{\left(
r\right)  }}\left(  p\right)  +3\zeta_{H^{\left(  r-1\right)  }}\left(
p\right)  +\zeta_{H^{\left(  r-2\right)  }}\left(  p\right)  -3\zeta
_{H^{\left(  r\right)  }}\left(  p-1\right) \\
&  \quad+\zeta_{H^{\left(  r\right)  }}\left(  p-2\right)  -2\zeta_{H^{\left(
r-1\right)  }}\left(  p-1\right)  .
\end{align*}
Later, G\"{o}ral and Sertba\c{s} \cite{GS} showed that the Euler sums of
generalized hyperharmonic numbers can be evaluated in terms of the Euler sums
of generalized harmonic numbers and special values of the Riemann zeta
function. However, their method does not determine the coefficients
explicitly. This gap is filled in this study. The following recurrence
relation for $H_{n}^{\left(  p,q\right)  }$ depending on the index $q$,%
\[
\left(  q-1\right)  H_{n}^{\left(  p,q\right)  }=\left(  n+q-1\right)
H_{n}^{\left(  p,q-1\right)  }-H_{n}^{\left(  p-1,q-1\right)  }%
\]
is obtained. Thanks to this recurrence relation, it is managed to obtain a
closed formula for $H_{n}^{\left(  p,q\right)  }$ in terms of $H_{n}^{\left(
p\right)  }$ in Theorem \ref{teo1}. This enables the evaluation of Euler sums
of generalized hyperharmonic numbers in terms of the Euler sums of generalized
harmonic numbers as
\[
\zeta_{H^{\left(  p,q+1\right)  }}\left(  r\right)  =\frac{1}{q!}\sum
_{m=0}^{q}\sum_{k=0}^{m}\left(  -1\right)  ^{k}%
%TCIMACRO{\QATOPD{[}{]}{q+1}{m+1}}%
%BeginExpansion
\genfrac{[}{]}{0pt}{}{q+1}{m+1}%
%EndExpansion
\binom{m}{k}\zeta_{H^{\left(  p-k\right)  }}\left(  r+k-m\right)  .
\]
A demonstration of this formula is the following example:%
\begin{align*}
2\zeta_{H^{\left(  6,4+1\right)  }}\left(  6\right)   &  =-1925\zeta\left(
11\right)  +\left(  175\pi^{2}-\frac{905}{4}-\frac{3937\pi^{8}}{544\,320}%
\right)  \zeta\left(  9\right) \\
&  +\left(  \frac{245\pi^{2}}{12}+\frac{35\pi^{4}}{18}+\frac{31\pi^{10}%
}{46\,656}\right)  \zeta\left(  7\right)  +\left(  \frac{\pi^{4}}{4}%
+\frac{5\pi^{6}}{1134}+\frac{31\pi^{12}}{6123\,600}\right)  \zeta\left(
5\right) \\
&  -\frac{35}{12}\zeta^{2}\left(  5\right)  -\frac{1}{3}\zeta\left(  3\right)
\zeta\left(  5\right)  +\frac{\pi^{6}}{1134}\zeta\left(  3\right)  +\frac
{\pi^{10}}{29\,160}+\frac{1406\pi^{12}}{638\,512\,875}.
\end{align*}
In addition, a counterpart of the reflection formula (\ref{rp}) is obtained in
the following form:
\[
\zeta_{H^{\left(  p,q+1\right)  }}\left(  r\right)  +\zeta_{H^{\left(
r,q+1\right)  }}\left(  p\right)  .
\]
Section two completes with this formula which serves to calculate sums similar
to the foregoing example with less computational cost.

In the last section we further extend our results. In this direction we
establish new and more general identities for the series whose terms are
generalizations of harmonic numbers and reciprocal binomial coefficients. For
instance,%
\[
\sum_{n=1}^{\infty}\frac{H_{n}^{\left(  p,q\right)  }}{\left(  n+m\right)
\binom{n+m+l}{l}}%
\]
is evaluated in terms of Riemann zeta values. This leads to several new
evaluation formulas for particular series involving generalized harmonic and
hyperharmonic numbers. We point out special cases of these formulas which
match with several known results in the literature.

\section{Euler sums of generalized hyperharmonic numbers}

In this section we present an evaluation formula for Euler sums $\zeta
_{H^{\left(  p,q\right)  }}\left(  r\right)  $ under certain conditions. To
state and prove our result we need some preliminaries.

Firstly, recall the polylogarithm defined by
\[
Li_{p}\left(  t\right)  =\sum_{k=1}^{\infty}\frac{t^{k}}{k^{p}},\text{
}(\left\vert t\right\vert \leq1\text{ if }p>1,\text{ and }\left\vert
t\right\vert <1\text{ if }p\leq1)\text{.}%
\]
The generating function of the numbers $H_{n}^{\left(  p,q\right)  }$ in terms
of the polylogarithm is \cite{DiMeCe}
\begin{equation}
\sum_{n=0}^{\infty}H_{n}^{\left(  p,q\right)  }t^{n}=\frac{Li_{p}\left(
t\right)  }{\left(  1-t\right)  ^{q}},\text{ }\left\vert t\right\vert
<1,\text{ }p,q\in\mathbb{Z}. \label{ghhgf}%
\end{equation}

Our first result presents the following reduction formula for $H_{n}^{\left(
p,q\right)  }$.

\begin{lemma}
Let $p$ and $q$ be integers with $q\geq1.$ Reduction relation for
$H_{n}^{\left(  p,q\right)  }$ in the index $q$\ is
\begin{equation}
\left(  q-1\right)  H_{n}^{\left(  p,q\right)  }=\left(  n+q-1\right)
H_{n}^{\left(  p,q-1\right)  }-H_{n}^{\left(  p-1,q-1\right)  }. \label{red}%
\end{equation}

\end{lemma}

\begin{proof}
We define the polynomial $H_{n}^{\left(  p,q\right)  }\left(  z\right)  $ as%
\[
H_{n}^{\left(  p,q\right)  }\left(  z\right)  =\sum_{k=0}^{n}H_{k}^{\left(
p,q\right)  }z^{k}.
\]
Considering (\ref{ghhgf}), we obtain the ordinary generating function of
$H_{n}^{\left(  p,q\right)  }\left(  z\right)  $ as%
\begin{equation}
\sum_{n=0}^{\infty}H_{n}^{\left(  p,q\right)  }\left(  z\right)  t^{n}%
=\frac{Li_{p}\left(  zt\right)  }{\left(  1-t\right)  \left(  1-zt\right)
^{q}}. \label{Hgf}%
\end{equation}
From (\ref{Hgf}), it can be seen that
\begin{equation}
z\frac{d}{dz}H_{n}^{\left(  p,q\right)  }\left(  z\right)  =H_{n}^{\left(
p-1,q\right)  }\left(  z\right)  +qzH_{n-1}^{\left(  p,q+1\right)  }\left(
z\right)  . \label{tur1}%
\end{equation}
On the other hand, we utilize (\ref{0}) twice to find that
\begin{align*}
H_{n}^{\left(  p,q\right)  }  &  =\sum_{k=1}^{n}H_{k}^{\left(  p,q-1\right)
}=\sum_{k=1}^{n}\sum_{j=1}^{k}H_{j}^{\left(  p,q-2\right)  }\\
&  =\left(  n+1\right)  H_{n}^{\left(  p,q-1\right)  }-\sum_{j=1}^{n}%
jH_{j}^{\left(  p,q-2\right)  }\\
&  =\left(  n+1\right)  H_{n}^{\left(  p,q-1\right)  }-\left.  \frac{d}%
{dz}H_{n}^{\left(  p,q-2\right)  }\left(  z\right)  \right\vert _{z=1},
\end{align*}
or equivalently%
\begin{equation}
\left.  \frac{d}{dz}H_{n}^{\left(  p,q-2\right)  }\left(  z\right)
\right\vert _{z=1}=\left(  n+1\right)  H_{n}^{\left(  p,q-1\right)  }%
-H_{n}^{\left(  p,q\right)  }\overset{\text{(\ref{0})}}{=}nH_{n}^{\left(
p,q-1\right)  }-H_{n-1}^{\left(  p,q\right)  }. \label{tur2}%
\end{equation}
Therefore, (\ref{tur1}) and (\ref{tur2}) yield the desired formula.
\end{proof}

The objective here is to express $H_{n}^{\left(  p,q\right)  }$ in terms of
$H_{n}^{\left(  p\right)  }.$ In \cite{DiMeCe} this relation is listed for at
most $q=4$ due to the complexity of the process. However, the next result
provides a general solution to this problem where the numbers $H_{n}^{\left(
p,q\right)  }$ are expressed in terms of the numbers $H_{n}^{\left(  p\right)
}$ and $%
%TCIMACRO{\QATOPD{[}{]}{q}{j}}%
%BeginExpansion
\genfrac{[}{]}{0pt}{}{q}{j}%
%EndExpansion
_{r}$. Here $%
%TCIMACRO{\QATOPD{[}{]}{q}{j}}%
%BeginExpansion
\genfrac{[}{]}{0pt}{}{q}{j}%
%EndExpansion
_{r}$ denotes the $r$-Stirling number of the first kind defined by the
"horizontal" generating function \cite{Broder,Ca,RM}%
\begin{equation}
\left(  x+r\right)  \left(  x+r+1\right)  \cdots\left(  x+r+q-1\right)
=\sum\limits_{j=0}^{q}%
%TCIMACRO{\QATOPD{[}{]}{q}{j}}%
%BeginExpansion
\genfrac{[}{]}{0pt}{}{q}{j}%
%EndExpansion
_{r}x^{j}. \label{rS1}%
\end{equation}
The essence of the theorem's proof is based on the relationship between
$r$-Stirling numbers and symmetric polynomials. The $k$th elementary symmetric
polynomial $e_{k}\left(  X_{1},\ldots,X_{q}\right)  $ in variables
$X_{1},\ldots,X_{q}$ is defined by (see for example \cite{Ma})
\begin{align*}
e_{0}\left(  X_{1},\ldots,X_{q}\right)   &  =1,\\
e_{k}\left(  X_{1},\ldots,X_{q}\right)   &  =\sum_{1\leq j_{1}<j_{2}%
<\cdots<j_{k}\leq q}\prod\limits_{i=1}^{k}X_{j_{i}},\text{ }1\leq k\leq q,\\
e_{k}\left(  X_{1},\ldots,X_{q}\right)   &  =0,\text{ }k>q,
\end{align*}
and possess the identity
\begin{equation}
\prod\limits_{j=1}^{q}\left(  x-X_{j}\right)  =\sum\limits_{j=0}^{q}\left(
-1\right)  ^{r-j}e_{q-j}\left(  X_{1},\ldots,X_{q}\right)  x^{j}. \label{**}%
\end{equation}
The comparison of (\ref{rS1}) with (\ref{**}) obviously leads to the following
relationship \cite[Theorem 4.1]{RM}
\begin{equation}
e_{q-j}\left(  n+1,n+2,\ldots,n+q\right)  =%
%TCIMACRO{\QATOPD{[}{]}{q}{j}}%
%BeginExpansion
\genfrac{[}{]}{0pt}{}{q}{j}%
%EndExpansion
_{n+1}. \label{2}%
\end{equation}

\begin{theorem}
\label{teo1}Let $p$ and $q$ be integers with $q\geq0.$ Then,%
\begin{equation}
q!H_{n}^{\left(  p,q+1\right)  }=\sum_{k=0}^{q}\left(  -1\right)  ^{k}%
%TCIMACRO{\QATOPD{[}{]}{q}{k}}%
%BeginExpansion
\genfrac{[}{]}{0pt}{}{q}{k}%
%EndExpansion
_{n+1}H_{n}^{\left(  p-k\right)  }. \label{4.2}%
\end{equation}

\end{theorem}

\begin{proof}
We employ (\ref{red}) on the right-hand side of%
\[
\left(  q-1\right)  qH_{n}^{\left(  p,q+1\right)  }=\left(  n+q\right)
\left(  q-1\right)  H_{n}^{\left(  p,q\right)  }-\left(  q-1\right)
H_{n}^{\left(  p-1,q\right)  },
\]
and see that%
\begin{align*}
\left(  q-1\right)  qH_{n}^{\left(  p,q+1\right)  }  &  =H_{n}^{\left(
p,q-1\right)  }\left\{  \left(  n+q\right)  \left(  n+q-1\right)  \right\} \\
&  \quad-H_{n}^{\left(  p-1,q-1\right)  }\left\{  \left(  n+q\right)  +\left(
n+q-1\right)  \right\}  +H_{n}^{\left(  p-2,q-1\right)  }\\
&  =\sum_{k=0}^{2}\left(  -1\right)  ^{k}e_{2-k}\left(  n+q-1,n+q\right)
H_{n}^{\left(  p-k,q+1-2\right)  }.
\end{align*}
These initial steps suggest that the following equality should hold:
\begin{align}
& \left(  q+1-r\right)  \left(  q+1-\left(  r-1\right)  \right)  \cdots\left(
q-1\right)  qH_{n}^{\left(  p,q+1\right)  }\label{*}\\
& =\sum_{k=0}^{r}\left(  -1\right)  ^{k}e_{r-k}\left(  n+q-\left(  r-1\right)
,n+q-\left(  r-2\right)  ,\ldots,n+q\right)  H_{n}^{\left(  p-k,q+1-r\right)
}. \nonumber
\end{align}
To prove this by induction we show that it is also true for $r+1\leq q.$ We
multiply (\ref{*}) by $\left(  q-r\right)  $ and then use (\ref{red}). Hence
we find that
\begin{align*}
&  \left(  q-r\right)  \left(  q+1-r\right)  \left(  q+1-\left(  r-1\right)
\right)  \cdots\left(  q-1\right)  qH_{n}^{\left(  p,q+1\right)  }\\
&  \ =\sum_{k=0}^{r}\left(  -1\right)  ^{k}e_{r-k}\left(  n+q-\left(
r-1\right)  ,\ldots,n+q\right)  \left(  n+q-r\right)  H_{n}^{\left(
p-k,q-r\right)  }\\
&  \quad+\sum_{k=0}^{r}\left(  -1\right)  ^{k+1}e_{r-k}\left(  n+q-\left(
r-1\right)  ,\ldots,n+q\right)  H_{n}^{\left(  p-k-1,q-r\right)  }\\
&  \ =e_{r}\left(  n+q-\left(  r-1\right)  ,\ldots,n+q\right)  \left(
n+q-r\right)  H_{n}^{\left(  p,q-r\right)  }\\
&  \quad+\sum_{k=1}^{r}\left(  -1\right)  ^{k}H_{n}^{\left(  p-k,q-r\right)
}\left\{  \left(  n+q-r\right)  e_{r-k}\left(  n+q-\left(  r-1\right)
,\ldots,n+q\right)  \right. \\
&  \qquad\qquad\qquad\qquad\qquad\qquad\qquad\left.  +e_{r+1-k}\left(
n+q-\left(  r-1\right)  ,\ldots,n+q\right)  \right\} \\
&  \quad+\left(  -1\right)  ^{r+1}e_{0}\left(  n+q-r,\ldots,n+q\right)
H_{n}^{\left(  p-\left(  r+1\right)  ,q-r\right)  }\\
&  \ =\sum_{k=0}^{r+1}\left(  -1\right)  ^{k}e_{r+1-k}\left(  n+q-r,\ldots
,n+q\right)  H_{n}^{\left(  p-k,q+1-\left(  r+1\right)  \right)  }.
\end{align*}
The case $r=q$ in (\ref{*}) gives
\[
q!H_{n}^{\left(  p,q+1\right)  }=\sum_{k=0}^{q}\left(  -1\right)  ^{k}%
e_{q-k}\left(  n+1,\ldots,n+q\right)  H_{n}^{\left(  p-k\right)  },
\]
which combines with (\ref{2}) to give the statement.
\end{proof}

Now, we are ready to state and prove our evaluation formula for $\zeta
_{H^{\left(  p,q\right)  }}\left(  r\right)  .$ Thanks to this formula the
evaluation of Euler sums of generalized hyperharmonic numbers reduces to the
evaluation of Euler sums of generalized harmonic numbers.

\begin{theorem}
\label{GHHES}For $p,q\geq1$ and $r>q+1,$ we have%
\[
\zeta_{H^{\left(  p,q+1\right)  }}\left(  r\right)  =\frac{1}{q!}\sum
_{m=0}^{q}\sum_{k=0}^{m}\left(  -1\right)  ^{k}%
%TCIMACRO{\QATOPD{[}{]}{q+1}{m+1}}%
%BeginExpansion
\genfrac{[}{]}{0pt}{}{q+1}{m+1}%
%EndExpansion
\binom{m}{k}\zeta_{H^{\left(  p-k\right)  }}\left(  r+k-m\right)  .
\]

\end{theorem}

\begin{proof}
From (\ref{4.2}) and the following identity \cite[p. 1661]{NR}
\[%
%TCIMACRO{\QATOPD{[}{]}{n}{k}}%
%BeginExpansion
\genfrac{[}{]}{0pt}{}{n}{k}%
%EndExpansion
_{r+1}=\sum_{m=k}^{n}%
%TCIMACRO{\QATOPD{[}{]}{n+1}{m+1}}%
%BeginExpansion
\genfrac{[}{]}{0pt}{}{n+1}{m+1}%
%EndExpansion
\binom{m}{k}r^{m-k},
\]
we have
\begin{align*}
H_{n}^{\left(  p,q+1\right)  }  &  =\frac{1}{q!}\sum_{k=0}^{q}\left(
-1\right)  ^{k}%
%TCIMACRO{\QATOPD{[}{]}{q}{k}}%
%BeginExpansion
\genfrac{[}{]}{0pt}{}{q}{k}%
%EndExpansion
_{n+1}H_{n}^{\left(  p-k\right)  }\\
&  =\frac{1}{q!}\sum_{k=0}^{q}\sum_{m=k}^{q}\left(  -1\right)  ^{k}%
%TCIMACRO{\QATOPD{[}{]}{q+1}{m+1}}%
%BeginExpansion
\genfrac{[}{]}{0pt}{}{q+1}{m+1}%
%EndExpansion
\binom{m}{k}n^{m-k}H_{n}^{\left(  p-k\right)  }.
\end{align*}
Multiplying both sides with $n^{-r}$ and summing over $n$ from $1$ to $\infty
$, we deduce the desired result.
\end{proof}

As mentioned introductory the sums $\zeta_{H^{\left(  p,q\right)  }}\left(
r\right)  $ were listed up to $q=3$ in \cite{DiMeCe}. With the help of Theorem
\ref{GHHES} these sums can be evaluated for further choices of $q.$ For
instance for $q=4$ one can obtain:%
\begin{align*}
\zeta_{H^{\left(  p,4\right)  }}\left(  r\right)   &  =\zeta_{H^{\left(
p\right)  }}\left(  r\right)  +\frac{11}{6}\zeta_{H^{\left(  p\right)  }%
}\left(  r-1\right)  +\zeta_{H^{\left(  p\right)  }}\left(  r-2\right)
+\frac{1}{6}\zeta_{H^{\left(  p\right)  }}\left(  r-3\right) \\
&  \quad-\frac{11}{6}\zeta_{H^{\left(  p-1\right)  }}\left(  r\right)
-2\zeta_{H^{\left(  p-1\right)  }}\left(  r-1\right)  -\frac{1}{2}%
\zeta_{H^{\left(  p-1\right)  }}\left(  r-2\right)  +\zeta_{H^{\left(
p-2\right)  }}\left(  r\right) \\
&  \quad+\frac{1}{2}\zeta_{H^{\left(  p-2\right)  }}\left(  r-1\right)
-\frac{1}{6}\zeta_{H^{\left(  p-3\right)  }}\left(  r\right)  .
\end{align*}
Hence, with the use of some values of $\zeta_{H^{\left(  p\right)  }}\left(
r\right)  $ listed in forthcoming Remark \ref{tablo}, a few concrete
expressions of $\zeta_{H^{\left(  p,4\right)  }}\left(  r\right)  $ are:\
\begin{align*}
\bullet\text{ }\zeta_{H^{\left(  1,4\right)  }}\left(  5\right)   &
=\frac{11}{2}\zeta\left(  5\right)  -\left(  1-\frac{11}{36}\pi^{2}\right)
\zeta\left(  3\right)  -\frac{1}{2}\left(  \zeta\left(  3\right)  \right)
^{2}-\frac{11}{216}\pi^{2}-\frac{\pi^{4}}{810}+\frac{\pi^{6}}{540},\\
\bullet\text{ }\zeta_{H^{\left(  2,4\right)  }}\left(  5\right)   &
=-10\zeta\left(  7\right)  +\left(  \frac{5}{6}\pi^{2}-\frac{21}{2}\right)
\zeta\left(  5\right)  +\left(  \frac{\pi^{4}}{45}+\frac{5}{6}\pi^{2}+\frac
{5}{12}\right)  \zeta\left(  3\right) \\
&  \quad+\frac{11}{4}\left(  \zeta\left(  3\right)  \right)  ^{2}+\frac
{7\pi^{4}}{1080}-\frac{55\pi^{6}}{13608},\\
\bullet\text{ }\zeta_{H^{\left(  3,4\right)  }}\left(  5\right)   &
=\zeta_{H^{\left(  3\right)  }}\left(  5\right)  +\frac{154}{3}\zeta\left(
7\right)  +\left(  \frac{14}{3}-\frac{55}{12}\pi^{2}\right)  \zeta\left(
5\right)  -2\left(  \zeta\left(  3\right)  \right)  ^{2}\\
&  \quad-\left(  \frac{7\pi^{2}}{18}+\frac{11\pi^{4}}{270}\right)
\zeta\left(  3\right)  -\frac{\pi^{4}}{540}+\frac{\pi^{6}}{324},\\
\bullet\text{ }\zeta_{H^{\left(  4,4\right)  }}\left(  5\right)   &
=-\frac{11}{6}\zeta_{H^{\left(  3\right)  }}\left(  5\right)  -\frac{125}%
{2}\zeta\left(  9\right)  +\left(  \frac{35}{6}\pi^{2}-63\right)  \zeta\left(
7\right) \\
&  \quad+\left(  \frac{35}{6}\pi^{2}+\frac{\pi^{4}}{18}\right)  \zeta\left(
5\right)  +\frac{\pi^{4}}{30}\zeta\left(  3\right)  -\frac{\pi^{6}}%
{1944}+\frac{143\pi^{8}}{680400},\\
\bullet\text{ }\zeta_{H^{\left(  5,4\right)  }}\left(  5\right)   &
=231\zeta\left(  9\right)  +\left(  21-\frac{385}{18}\pi^{2}\right)
\zeta\left(  7\right)  -\left(  \frac{11}{60}\pi^{4}+\frac{23}{12}\pi
^{2}\right)  \zeta\left(  5\right) \\
&  \quad+\frac{1}{2}\left(  \zeta\left(  5\right)  \right)  ^{2}+\zeta\left(
3\right)  \zeta\left(  5\right)  -\frac{7\pi^{4}}{540}\zeta\left(  3\right)
-\frac{\pi^{8}}{8100}+\frac{\pi^{10}}{187110}.
\end{align*}
\textbf{\ }

The following corollary gives the reflection formula for Euler sums of
generalized hyperharmonic numbers. Combined with (\ref{GES}), this corollary
shows that $\zeta_{H^{\left(  p,q+1\right)  }}\left(  r\right)  +\zeta
_{H^{\left(  r,q+1\right)  }}\left(  p\right)  $ can be written as a
combination of Riemann zeta values. In this way, particular Euler sums of type
$\zeta_{H^{\left(  p,q\right)  }}\left(  p\right)  $ can be evaluated with
less computation.

\begin{corollary}
\label{cor1}Let $p>q+1,$ $r>q+1$ and $p+r$ be even. Then%
\begin{align*}
&  \zeta_{H^{\left(  p,q+1\right)  }}\left(  r\right)  +\zeta_{H^{\left(
r,q+1\right)  }}\left(  p\right) \\
&  =\zeta\left(  p+r\right)  +\frac{2}{q!}\sum_{\substack{m=0\\m\text{ odd}%
}}^{q}\sum_{k=0}^{m}\left(  -1\right)  ^{k}%
%TCIMACRO{\QATOPD{[}{]}{q+1}{m+1}}%
%BeginExpansion
\genfrac{[}{]}{0pt}{}{q+1}{m+1}%
%EndExpansion
\binom{m}{k}\zeta_{H^{\left(  p-k\right)  }}\left(  r+k-m\right) \\
&  +\frac{1}{q!}\sum_{m=0}^{q}\sum_{k=0}^{m}\left(  -1\right)  ^{m+k}%
%TCIMACRO{\QATOPD{[}{]}{q+1}{m+1}}%
%BeginExpansion
\genfrac{[}{]}{0pt}{}{q+1}{m+1}%
%EndExpansion
\binom{m}{k}\zeta\left(  p-k\right)  \zeta\left(  r+k-m\right)  .
\end{align*}

\end{corollary}

\begin{proof}
Let $\left(  p+r\right)  $ be even. It is obvious from Theorem \ref{GHHES}
that%
\begin{align*}
&  \zeta_{H^{\left(  p,q+1\right)  }}\left(  r\right)  +\zeta_{H^{\left(
r,q+1\right)  }}\left(  p\right) \\
&  =\frac{1}{q!}\sum_{m=0}^{q}\sum_{k=0}^{m}\left(  -1\right)  ^{k}%
%TCIMACRO{\QATOPD{[}{]}{q+1}{m+1}}%
%BeginExpansion
\genfrac{[}{]}{0pt}{}{q+1}{m+1}%
%EndExpansion
\binom{m}{k}\left\{  \sum_{n=1}^{\infty}\frac{H_{n}^{\left(  p-k\right)  }%
}{n^{r+k-m}}+\left(  -1\right)  ^{m}\sum_{n=1}^{\infty}\frac{H_{n}^{\left(
r+k-m\right)  }}{n^{p-k}}\right\}  .
\end{align*}
We write the right-hand side as
\begin{align*}
&  \sum_{\substack{m=0\\m\text{ odd}}}^{q}\sum_{k=0}^{m}\left(  -1\right)
^{k}%
%TCIMACRO{\QATOPD{[}{]}{q+1}{m+1}}%
%BeginExpansion
\genfrac{[}{]}{0pt}{}{q+1}{m+1}%
%EndExpansion
\binom{m}{k}\left\{  \sum_{n=1}^{\infty}\frac{H_{n}^{\left(  p-k\right)  }%
}{n^{r+k-m}}-\sum_{n=1}^{\infty}\frac{H_{n}^{\left(  r+k-m\right)  }}{n^{p-k}%
}\right\} \\
&  +\sum_{0\leq m\leq q/2}\sum_{k=0}^{2m}\left(  -1\right)  ^{k}%
%TCIMACRO{\QATOPD{[}{]}{q+1}{2m+1}}%
%BeginExpansion
\genfrac{[}{]}{0pt}{}{q+1}{2m+1}%
%EndExpansion
\binom{2m}{k}\left\{  \sum_{n=1}^{\infty}\frac{H_{n}^{\left(  p-k\right)  }%
}{n^{r+k-2m}}+\sum_{n=1}^{\infty}\frac{H_{n}^{\left(  r+k-2m\right)  }%
}{n^{p-k}}\right\}  .
\end{align*}
By the reflection formula (\ref{rp}) we have
\[
\sum_{n=1}^{\infty}\frac{H_{n}^{\left(  p-k\right)  }}{n^{r+k-2m}}+\sum
_{n=1}^{\infty}\frac{H_{n}^{\left(  r+k-2m\right)  }}{n^{p-k}}=\zeta\left(
p+r-2m\right)  +\zeta\left(  p-k\right)  \zeta\left(  r+k-2m\right)  .
\]
Moreover, for odd $m,$ it can be seen from (\ref{GES})\ that%
\begin{align*}
\sum_{n=1}^{\infty}\frac{H_{n}^{\left(  p-k\right)  }}{n^{r+k-m}}-\sum
_{n=1}^{\infty}\frac{H_{n}^{\left(  r+k-m\right)  }}{n^{p-k}}  &
=2\zeta_{H^{\left(  p-k\right)  }}\left(  r+k-m\right)  -\zeta\left(
p+r-m\right) \\
&  \quad-\zeta\left(  p-k\right)  \zeta\left(  r+k-m\right)  .
\end{align*}
Hence, we obtain the desired equation.
\end{proof}

\begin{remark}
\label{tablo}For interested readers we would like to list some values of
$\zeta_{H^{\left(  p\right)  }}\left(  r\right)  $, used in the evaluations of
$\zeta_{H^{\left(  p,4\right)  }}\left(  5\right)  ,$ $1\leq p\leq5,$ and
$\zeta_{H^{\left(  6,5\right)  }}\left(  6\right)  .$ These are calculated
with the help of (\ref{ES}), (\ref{GES}) and (\ref{rp}).

\noindent\hspace{-0.09in}%
\begin{tabular}
[c]{ll}%
$\bullet\text{ }\zeta_{H^{\left(  1\right)  }}\left(  2\right)  =2\zeta\left(
3\right)  ,$ & $\bullet\text{ }\zeta_{H^{\left(  3\right)  }}\left(  6\right)
=\frac{85}{2}\zeta\left(  9\right)  -\frac{7\pi^{2}}{2}\zeta\left(  7\right)
-\frac{\pi^{4}}{15}\zeta\left(  5\right)  ,$\medskip\\
$\bullet\text{ }\zeta_{H^{\left(  1\right)  }}\left(  3\right)  =\frac{\pi
^{4}}{72},$ & $\bullet\text{ }\zeta_{H^{\left(  4\right)  }}\left(  2\right)
=-\zeta^{2}\left(  3\right)  +\frac{37\pi^{6}}{11340},$\medskip\\
$\bullet\text{ }\zeta_{H^{\left(  1\right)  }}\left(  4\right)  =3\zeta\left(
5\right)  -\frac{\pi^{2}}{6}\zeta\left(  3\right)  ,$ & $\bullet\text{ }%
\zeta_{H^{\left(  4\right)  }}\left(  3\right)  =-17\zeta\left(  7\right)
+\frac{5\pi^{2}}{3}\zeta\left(  5\right)  +\frac{\pi^{4}}{90}\zeta\left(
3\right)  ,$\medskip\\
$\bullet\text{ }\zeta_{H^{\left(  1\right)  }}\left(  5\right)  =-\frac{1}%
{2}\zeta^{2}\left(  3\right)  +\frac{\pi^{6}}{540},$ & $\bullet\text{ }%
\zeta_{H^{\left(  4\right)  }}\left(  4\right)  =\frac{13\pi^{8}}{113\,400}%
,$\medskip\\
$\bullet\text{ }\zeta_{H^{\left(  2\right)  }}\left(  2\right)  =\frac
{7\pi^{4}}{360},$ & $\bullet\text{ }\zeta_{H^{\left(  4\right)  }}\left(
5\right)  =-\frac{125}{2}\zeta\left(  9\right)  +\frac{35\pi^{2}}{6}%
\zeta\left(  7\right)  +\frac{\pi^{4}}{18}\zeta\left(  5\right)  ,$\medskip\\
$\bullet\text{ }\zeta_{H^{\left(  2\right)  }}\left(  3\right)  =-\frac{9}%
{2}\zeta\left(  5\right)  +\frac{\pi^{2}}{2}\zeta\left(  3\right)  ,$ &
$\bullet\text{ }\zeta_{H^{\left(  5\right)  }}\left(  2\right)  =11\zeta
\left(  7\right)  -\frac{2\pi^{2}}{3}\zeta\left(  5\right)  -\frac{\pi^{4}%
}{45}\zeta\left(  3\right)  ,$\medskip\\
$\bullet\text{ }\zeta_{H^{\left(  2\right)  }}\left(  4\right)  =\zeta
^{2}\left(  3\right)  -\frac{\pi^{6}}{2835},$ & $\bullet\text{ }%
\zeta_{H^{\left(  5\right)  }}\left(  4\right)  =\frac{127}{2}\zeta\left(
9\right)  -\frac{35\pi^{2}}{6}\zeta\left(  7\right)  -\frac{2\pi^{4}}{45}%
\zeta\left(  5\right)  ,$\medskip\\
$\bullet\text{ }\zeta_{H^{\left(  2\right)  }}\left(  5\right)  =-10\zeta
\left(  7\right)  +\frac{5\pi^{2}}{6}\zeta\left(  5\right)  $ & $\bullet\text{
}\zeta_{H^{\left(  5\right)  }}\left(  5\right)  =\frac{1}{2}\zeta^{2}\left(
5\right)  +\frac{\pi^{10}}{187110},$\medskip\\
$\qquad\qquad\qquad+\frac{\pi^{4}}{45}\zeta\left(  3\right)  ,$ &
$\bullet\text{ }\zeta_{H^{\left(  5\right)  }}\left(  6\right)  =\frac{463}%
{2}\zeta\left(  11\right)  -21\pi^{2}\zeta\left(  9\right)  -\frac{7}{30}%
\pi^{4}\zeta\left(  7\right)  ,$\medskip\\
$\bullet\text{ }\zeta_{H^{\left(  3\right)  }}\left(  2\right)  =\frac{11}%
{2}\zeta\left(  5\right)  -\frac{\pi^{2}}{3}\zeta\left(  3\right)  ,$ &
$\bullet\text{ }\zeta_{H^{\left(  6\right)  }}\left(  3\right)  =\frac
{7\pi^{2}}{2}\zeta\left(  7\right)  -\frac{83}{2}\zeta\left(  9\right)
+\frac{\pi^{4}}{15}\zeta\left(  5\right)  $\\
$\bullet\text{ }\zeta_{H^{\left(  3\right)  }}\left(  3\right)  =\frac{1}%
{2}\zeta^{2}\left(  3\right)  +\frac{\pi^{6}}{1890},$ & $\qquad\qquad
\qquad+\frac{\pi^{6}}{945}\zeta\left(  3\right)  ,$\medskip\\
$\bullet\text{ }\zeta_{H^{\left(  3\right)  }}\left(  4\right)  =18\zeta
\left(  7\right)  -\frac{5\pi^{2}}{3}\zeta\left(  5\right)  ,$ &
$\bullet\text{ }\zeta_{H^{\left(  6\right)  }}\left(  5\right)  =21\pi
^{2}\zeta\left(  9\right)  -\frac{461}{2}\zeta\left(  11\right)  +\frac
{7\pi^{4}}{30}\zeta\left(  7\right)  $\\
& $\qquad\qquad\qquad+\frac{\pi^{6}}{945}\zeta\left(  5\right)  .$%
\end{tabular}

\end{remark}

\section{\textbf{Series involving harmonic numbers and reciprocal binomial
coefficients}}

In this section, we introduce evaluation formulas for some series involving
the harmonic numbers and their generalizations.

\begin{theorem}
\label{pro1}Let $p\geq1$ and $q,$ $l\geq0$ be integers with $l\geq q$. For
$m\geq1$,
\begin{align}
&  \sum_{n=1}^{\infty}\frac{H_{n}^{\left(  p,q\right)  }}{\left(  n+m\right)
\binom{n+m+l}{l}}\label{L5}\\
&  \ =\sum_{j=0}^{l-q}\binom{l-q}{j}\left\{  \frac{\left(  -1\right)
^{j+p-1}H_{m+j}}{\left(  m+j\right)  ^{p}}+\sum_{k=1}^{p-1}\frac{\left(
-1\right)  ^{j+k-1}}{\left(  m+j\right)  ^{k}}\zeta\left(  p+1-k\right)
\right\}  \nonumber
\end{align}
and
\begin{align}
&  \sum_{n=1}^{\infty}\frac{H_{n}^{\left(  p,q\right)  }}{n\binom{n+l}{l}%
}\label{L6}\\
&  \ =\zeta\left(  p+1\right)  -\sum_{j=1}^{l-q}\binom{l-q}{j}\left\{
\frac{\left(  -1\right)  ^{j+p}H_{j}}{j^{p}}+\sum_{k=1}^{p-1}\frac{\left(
-1\right)  ^{j+k}}{j^{k}}\zeta\left(  p+1-k\right)  \right\}  . \nonumber
\end{align}

\end{theorem}

\begin{proof}
Using the formula (see \cite[p.909]{JZ})%
\[%
%TCIMACRO{\dint \limits_{0}^{1}}%
%BeginExpansion
{\displaystyle\int\limits_{0}^{1}}
%EndExpansion
t^{n+m-1}\left(  1-t\right)  ^{l}dt=\frac{1}{\left(  n+m\right)  \binom
{n+m+l}{l}},
\]
we can write%
\[
\frac{H_{n}^{\left(  p,q\right)  }}{\left(  n+m\right)  \binom{n+m+l}{l}}=%
%TCIMACRO{\dint \limits_{0}^{1}}%
%BeginExpansion
{\displaystyle\int\limits_{0}^{1}}
%EndExpansion
H_{n}^{\left(  p,q\right)  }t^{n+m-1}\left(  1-t\right)  ^{l}dt.
\]
With the help of (\ref{ghhgf}), we get%
\begin{align*}
\sum_{n=1}^{\infty}\frac{H_{n}^{\left(  p,q\right)  }}{\left(  n+m\right)
\binom{n+m+l}{l}}  &  =%
%TCIMACRO{\dint \limits_{0}^{1}}%
%BeginExpansion
{\displaystyle\int\limits_{0}^{1}}
%EndExpansion
t^{m-1}\left(  1-t\right)  ^{l-q}Li_{p}\left(  t\right)  dt\\
&  =\sum_{j=0}^{l-q}\binom{l-q}{j}\left(  -1\right)  ^{j}\sum_{n=1}^{\infty
}\frac{1}{n^{p}\left(  n+m+j\right)  }.
\end{align*}
Then (\ref{L5}) follows from (\ref{mu}). If $m=0,$ then we have%
\[
\sum_{n=1}^{\infty}\frac{H_{n}^{\left(  p,q\right)  }}{n\binom{n+l}{l}}%
=\zeta\left(  p+1\right)  +\sum_{j=1}^{l-q}\binom{l-q}{j}\left(  -1\right)
^{j}\sum_{n=1}^{\infty}\frac{1}{n^{p}\left(  n+j\right)  },
\]
which is equivalent to (\ref{L6}).
\end{proof}

Now, we deal with some special cases of Theorem \ref{pro1}. Setting $q=l$
gives
\begin{equation}
\sum_{n=1}^{\infty}\frac{H_{n}^{\left(  p,q\right)  }}{\left(  n+m\right)
\binom{n+m+q}{q}}=\frac{\left(  -1\right)  ^{p-1}H_{m}}{m^{p}}+\sum
_{i=1}^{p-1}\frac{\left(  -1\right)  ^{i-1}}{m^{i}}\zeta\left(  p+1-i\right)
\label{17}%
\end{equation}
and
\begin{equation}
\sum_{n=1}^{\infty}\frac{H_{n}^{\left(  p,q\right)  }}{n\binom{n+q}{q}}%
=\zeta\left(  p+1\right)  . \label{18}%
\end{equation}
Note that the variable $q$ does not appear in the right-hand sides and all
these series converge very slowly.

For $p=1,$\textbf{ }(\ref{17}) and (\ref{18}) give Proposition 5 and
Proposition 6 in \cite{DB}%
\[
\sum_{n=1}^{\infty}\frac{h_{n}^{\left(  q\right)  }}{n\binom{n+q}{q}}=\frac
{1}{6}\pi^{2} \text{ \ \ and \ \ }\sum_{n=1}^{\infty}\frac{h_{n-1}^{\left(
q\right)  }}{n\binom{n+q}{q}}=1,
\]
respectively. (\ref{17}) also yields \cite[Eq. (2.30)]{XZZ} for $q=1$.
Additionally, when $p=1$ in Theorem \ref{pro1}, we reach that%
\[
\sum_{n=1}^{\infty}\frac{h_{n}^{\left(  q\right)  }}{\left(  n+m\right)
\binom{n+m+l}{l}}=\sum_{j=0}^{l-q}\binom{l-q}{j}\left(  -1\right)  ^{j}%
\frac{H_{m+j}}{m+j}%
\]
and
\[
\sum_{n=1}^{\infty}\frac{h_{n}^{\left(  q\right)  }}{n\binom{n+l}{l}}=\frac
{1}{6}\pi^{2}+\sum_{j=1}^{l-q}\left(  -1\right)  ^{j}\binom{l-q}{j}\frac
{H_{j}}{j}.
\]
Now employing \cite[Eq.(18)]{CHU}
\[
\sum_{k=0}^{m}\left(  -1\right)  ^{k}\binom{m}{k}\frac{H_{n+k}}{n+k}%
=\frac{H_{n+m}-H_{m}}{n\binom{n+m}{m}}%
\]
and \cite[Eq. (9.4b)]{B2018}
\[
\sum_{j=1}^{n}\left(  -1\right)  ^{j+1}\binom{n}{j}\frac{H_{j}}{j}%
=H_{n}^{\left(  2\right)  }
\]
gives the following closed forms for series involving hyperharmonic numbers
with reciprocal binomial coefficients.

\begin{corollary}
Let $q,l\geq0$ be integers with $l\geq q$. For all integers $m\geq1$
\[
\sum_{n=1}^{\infty}\frac{h_{n}^{\left(  q\right)  }}{\left(  n+m\right)
\binom{n+m+l}{l}}=\frac{H_{m+l-q}-H_{l-q}}{m\binom{n+m+l-q}{l-q}}%
\]
and
\[
\sum_{n=1}^{\infty}\frac{h_{n}^{\left(  q\right)  }}{n\binom{n+l}{l}}=\frac
{1}{6}\pi^{2}-H_{l-q}^{\left(  2\right)  }.
\]

\end{corollary}

For $p=q=1,$ Theorem \ref{pro1} gives \cite[Eq. (2.31)]{SS}%
\[
\sum_{n=1}^{\infty}\frac{H_{n}}{\left(  n+m\right)  \binom{n+m+q}{q}}%
=\frac{H_{n+q-1}-H_{q-1}}{n\binom{n+q-1}{q-1}}%
\]
and \cite{SC}
\[
\sum_{n=1}^{\infty}\frac{H_{n}}{n\binom{n+l}{l}}=\frac{1}{6}\pi^{2}%
-H_{l-1}^{\left(  2\right)  }.
\]
For $q=1,$ (\ref{L6}) becomes
\begin{align}
&  \sum_{n=1}^{\infty}\frac{H_{n}^{\left(  p\right)  }}{n\binom{n+l}{l}%
}\label{sofo}\\
&  \ =\zeta\left(  p+1\right)  -\sum_{j=1}^{l-1}\binom{l-1}{j}\left\{
\frac{\left(  -1\right)  ^{j+p}H_{j}}{j^{p}}+\sum_{k=1}^{p-1}\frac{\left(
-1\right)  ^{j+k}}{j^{k}}\zeta\left(  p+1-k\right)  \right\}  , \nonumber
\end{align}
which is also given by Sofo \cite[Theorem 2.2]{S3} in a slightly different form.

Setting $q=1$ in (\ref{L5}) yields the following corollary involving
generalized harmonic numbers.

\begin{corollary}
For all integers $m,p,l\geq1,$
\begin{align}
&  \sum_{n=1}^{\infty}\frac{H_{n}^{\left(  p\right)  }}{\left(  n+m\right)
\binom{n+m+l}{l}}\label{26}\\
&  =\sum_{j=0}^{l-1}\binom{l-1}{j}\left(  -1\right)  ^{j}\left\{
\frac{\left(  -1\right)  ^{p-1}H_{m+j}}{\left(  m+j\right)  ^{p}}+\sum
_{k=1}^{p-1}\frac{\left(  -1\right)  ^{k-1}}{\left(  m+j\right)  ^{k}}%
\zeta\left(  p+1-k\right)  \right\}  . \nonumber
\end{align}

\end{corollary}

The following particular cases can be deduced setting $p=2$ and $p=3$:%
\begin{align}
\sum_{n=1}^{\infty}\frac{H_{n}^{\left(  2\right)  }}{\left(  n+m\right)
\binom{n+m+l}{l}}  &  =\frac{\pi^{2}}{6m\binom{m+l-1}{l-1}}-\sum_{j=0}%
^{l-1}\left(  -1\right)  ^{j}\binom{l-1}{j}\frac{H_{m+j}}{\left(  m+j\right)
^{2}},\label{L2}\\
\sum_{n=1}^{\infty}\frac{H_{n}^{\left(  3\right)  }}{\left(  n+m\right)
\binom{n+m+l}{l}}  &  =\sum_{j=0}^{l-1}\left(  -1\right)  ^{j}\binom{l-1}%
{j}\frac{H_{m+j}}{\left(  m+j\right)  ^{3}}\nonumber\\
&  +\frac{1}{m\binom{m+l-1}{l-1}} \left\{  \zeta\left(  3\right)  -\frac
{\pi^{2}}{6}\left(  H_{m+l-1}-H_{m-1}\right)  \right\}  .\nonumber
\end{align}

It is worth noting that the special case with choices $m=6$ and $l=3$ in
(\ref{L2}) is recorded in \cite[Remark 1]{S5} despite a misprint. The case is
below
\[
\sum_{n=1}^{\infty}\frac{H_{n}^{\left(  2\right)  }}{\left(  n+6\right)
\binom{n+9}{3}}=\frac{1}{168}\zeta\left(  2\right)  -\frac{37073}{7902720}.
\]

For our final results, we deal with the special case $q=2$ of Theorem
\ref{pro1}. By aid of (\ref{4.2}), $%
%TCIMACRO{\QATOPD{[}{]}{q}{0}}%
%BeginExpansion
\genfrac{[}{]}{0pt}{}{q}{0}%
%EndExpansion
_{r}=r\left(  r+1\right)  \cdots\left(  r+q-1\right)  $ and $%
%TCIMACRO{\QATOPD{[}{]}{q}{q}}%
%BeginExpansion
\genfrac{[}{]}{0pt}{}{q}{q}%
%EndExpansion
_{r}=1,$ we have%
\[
\sum_{n=1}^{\infty}\frac{H_{n}^{\left(  p\right)  }}{\binom{n+l}{l}}%
=\sum_{n=1}^{\infty}\frac{H_{n}^{\left(  p,2\right)  }}{n\binom{n+l}{l}}%
+\sum_{n=1}^{\infty}\frac{H_{n}^{\left(  p-1\right)  }}{n\binom{n+l}{l}}%
-\sum_{n=1}^{\infty}\frac{H_{n}^{\left(  p\right)  }}{n\binom{n+l}{l}},
\]
where $l$ is any integer greater than $1$. From (\ref{L6}), (\ref{sofo}) and
some arrangements we obtain
\[
\sum_{n=1}^{\infty}\frac{H_{n}^{\left(  p\right)  }}{\binom{n+l}{l}}%
=\zeta\left(  p\right)  +\sum_{j=1}^{l-1}\left(  -1\right)  ^{j}\left\{
\binom{l-1}{j}\mu\left(  p-1,j\right)  -\binom{l-2}{j-1}\mu\left(  p,j\right)
\right\}  ,
\]
where $\mu\left(  p,j\right)  $ is given in (\ref{mu}). A slightly different
form of the equation above is given in \cite[Theorem 2.1]{S3}. Similarly
\begin{align*}
\sum_{n=1}^{\infty}\frac{nH_{n}^{\left(  p\right)  }}{\left(  n+m\right)
\binom{n+m+l}{l}}  &  =\sum_{n=1}^{\infty}\frac{H_{n}^{\left(  p,2\right)  }%
}{\left(  n+m\right)  \binom{n+m+l}{l}}+\sum_{n=1}^{\infty}\frac
{H_{n}^{\left(  p-1\right)  }}{\left(  n+m\right)  \binom{n+m+l}{l}}\\
&  -\sum_{n=1}^{\infty}\frac{H_{n}^{\left(  p\right)  }}{\left(  n+m\right)
\binom{n+m+l}{l}}.
\end{align*}
Then exploiting (\ref{L6}) and (\ref{26}) in the last equation yields the
following corollary.

\begin{corollary}
Let $l>1$ be an integer. Then \textbf{ }%
\begin{align*}
&  \sum_{n=1}^{\infty}\frac{nH_{n}^{\left(  p\right)  }}{\left(  n+m\right)
\binom{n+m+l}{l}}\\
&  \ =\mu\left(  p-1,m\right)  +\sum_{j=1}^{l-1}\left(  -1\right)
^{j}\left\{  \binom{l-1}{j}\mu\left(  p-1,m+j\right)  -\binom{l-2}{j-1}%
\mu\left(  p,m+j\right)  \right\}  ,
\end{align*}
where $\mu\left(  p,j\right)  $ is given in (\ref{mu}).
\end{corollary}

For\textbf{ }$l=2$\textbf{ }this formula can be read as%
\begin{align*}
&  \sum_{n=1}^{\infty}\frac{nH_{n}^{\left(  p\right)  }}{\left(  n+m\right)
\left(  n+m+1\right)  \left(  n+m+2\right)  }\\
&  =\left(  -1\right)  ^{p+1}\frac{m+2}{2\left(  m+1\right)  ^{p}}%
H_{m+1}+\left(  -1\right)  ^{p}\frac{H_{m}}{2m^{p-1}}\\
&  \quad+\sum_{k=1}^{p-2}\left(  -1\right)  ^{k-1}\left\{  \frac{1}{2m^{k}%
}-\frac{1}{2\left(  m+1\right)  ^{k}}\right\}  \zeta\left(  p-k\right)
+\sum_{k=1}^{p-1}\frac{\left(  -1\right)  ^{k-1}}{2\left(  m+1\right)  ^{k}%
}\zeta\left(  p+1-k\right)  .
\end{align*}
The first few cases of this formula are listed below:%
\begin{align*}
\sum_{n=1}^{\infty}\frac{nH_{n}}{\left(  n+m\right)  \left(  n+m+1\right)
\left(  n+m+2\right)  }  &  =\frac{1}{2\left(  m+1\right)  }\left(
H_{m+1}+1\right)  ,\\
\sum_{n=1}^{\infty}\frac{nH_{n}^{\left(  2\right)  }}{\left(  n+m\right)
\left(  n+m+1\right)  \left(  n+m+2\right)  }  &  =\frac{\pi^{2}}{12\left(
m+1\right)  }-\frac{\left(  m+2\right)  }{2\left(  m+1\right)  ^{2}}%
H_{m+1}+\frac{1}{2m}H_{m},\\
\sum_{n=1}^{\infty}\frac{nH_{n}^{\left(  3\right)  }}{\left(  n+m\right)
\left(  n+m+1\right)  \left(  n+m+2\right)  }  &  =\frac{\zeta\left(
3\right)  }{2\left(  m+1\right)  }+\frac{\pi^{2}}{12m\left(  m+1\right)  ^{2}%
}\\
&  \ \ \ +\frac{m+2}{2\left(  m+1\right)  ^{3}}H_{m+1}-\frac{1}{2m^{2}}H_{m}.
\end{align*}

\end{document}